\documentclass[11pt,reqno]{amsart}
\usepackage{amsmath,amssymb,amsthm,amsfonts,enumerate,hyperref}
\usepackage{verbatim}
\usepackage{xcolor}
\usepackage{hyperref}
\hypersetup{
    colorlinks=true, %set true if you want colored links
    linktoc=all,     %set to all if you want both sections and subsections linked
    linkcolor=blue,  %choose some color if you want links to stand out
}
\usepackage{url}

\def\sideremark#1{\ifvmode\leavevmode\fi\vadjust{\vbox to0pt{\vss
\hbox to 0pt{\hskip\hsize\hskip1em%                          
\vbox{\hsize2cm\tiny\raggedright\pretolerance10000%        
\noindent {\color{red}{#1}}\hfill}\hss}\vbox to8pt{\vfil}\vss}}}%

\theoremstyle{plain}
\newtheorem{theorem}{Theorem}[section]
\newtheorem{theorem2}{Theorem}[section]
\newtheorem{proposition}[theorem2]{Proposition}

\newtheorem{lemma}[theorem2]{Lemma}

\theoremstyle{definition}
\newtheorem{definition}[theorem2]{Definition}

\theoremstyle{remark}

\numberwithin{equation}{section}

\newcommand{\vol}{\mathrm{Vol}}

\newcommand{\Ric}{\mathrm{Ric}}

\newcommand{\Hess}{\mathrm{Hess}}
\newcommand{\tr}{\mathrm{tr}}

\begin{document}
\title{Rigidity of $SU_n$-type symmetric spaces}
\author{Wafaa Batat}
\address{Ecole Nationale Polytechnique d'Oran, B.P 1523 El M'naouar, 31000 Oran, Algeria}
\email{batatwafa@yahoo.fr}
\author{Stuart James Hall}
\address{School of Mathematics and Statistics, Herschel Building, Newcastle University, Newcastle-upon-Tyne, NE1 7RU} 
\email{stuart.hall@ncl.ac.uk}
\urladdr{\href{https://www.ncl.ac.uk/maths-physics/staff/profile/stuarthall}{https://www.ncl.ac.uk/maths-physics/staff/profile/stuarthall}}
\author{Thomas Murphy}
\address{Department of Mathematics, California State University Fullerton, 800 N. State College Bld., Fullerton, CA 92831, USA.}
\email{tmurphy@fullerton.edu}
\urladdr{\href{http://www.fullerton.edu/math/faculty/tmurphy/}{http://www.fullerton.edu/math/faculty/tmurphy/}}
\author{James Waldron}
\address{School of Mathematics and Statistics, Herschel Building, Newcastle University, Newcastle-upon-Tyne, NE1 7RU} 
\email{james.waldron@ncl.ac.uk}
\urladdr{\href{https://www.ncl.ac.uk/maths-physics/staff/profile/jameswaldron}{https://www.ncl.ac.uk/maths-physics/staff/profile/jameswaldron}}
\maketitle  
\begin{abstract}
We prove that the bi-invariant Einstein metric on $SU_{2n+1}$ is isolated in the moduli space of Einstein metrics, even though it admits infinitesimal deformations. This gives a non-K\"ahler, non-product example of this phenomenon adding to the famous example of $\mathbb{CP}^{2n}\times\mathbb{CP}^{1}$ found by Koiso.  We apply our methods to derive similar solitonic rigidity results for the K\"ahler--Einstein metrics on `odd'  Grassmannians. We also make explicit a connection between non-integrable deformations and the dynamical instability of metrics under Ricci flow. 
\end{abstract}

\section{Introduction}\label{sec:intro}
\subsection{Rigidity of Einstein metrics}
\noindent For a fixed manifold $M$, a central object of study is the set of  Einstein metrics  viewed as a subset of all Riemannian metrics on $M$. In general, giving a full description of the set of Einstein metrics on  $M$ is an intractable project; a more reasonable goal is to produce a local description of this set around some fixed Einstein metric $g$.\\
\\
The foundational work on the  local structure theory of the set of Einstein metrics was carried out by Koiso  \cite{Koi_EMCS} who showed that the premoduli space of Einstein metrics about $g$ has the structure of an analytic subset of a smooth manifold $\mathcal{Z}$. Furthermore, the tangent space of $\mathcal{Z}$ at $g$ is naturally isomorphic to a certain eigenspace of the Lichnerowicz Laplacian. Hence, if the eigenspace is zero-dimensional, $g$ must be isolated in the premoduli space; isolated Einstein metrics are often referred to as being \textit{rigid}.  For spaces where the dimension of the eigenspace is positive, an interesting problem arises: do any of the tangent vectors (so-called {\it infinitesimal deformations} of the Einstein metric) come from genuine deformations of $g$ through Einstein metrics?\\
\\
Koiso investigated this question for compact symmetric spaces in  \cite{KoiOsaka1} and \cite{KoiOsaka2}. We will explain Koiso's findings in detail in Section \ref{sec:deform}. A key result is that all but one of the symmetric spaces where the dimension of the relevant eigenspace of the Lichnerowicz Laplacian is non-zero are associated to $SU_n$; they include the bi-invariant metric on $SU_n$ for $n>2$, and the K\"ahler--Einstein metric on the Grassmannians ${SU_n/S(U_{n-k}\times U_{k})}$ for $n>3$ and $k>1$. To investigate the rigidity of these metrics we need to investigate whether any infinitesimal deformations integrate up to genuine curves of Einstein metrics.  Koiso developed an obstruction to integrability and used this to show that the canonical metric on $\mathbb{CP}^{2n}\times \mathbb{CP}^{1}$ is isolated in the moduli space but admits infinitesimal deformations. This was the first such example of this phenomenon in the literature.\\
\\
One difficulty in computing Koiso's obstruction is finding a concrete description of the space of infinitesimal deformations beyond its formal definition as an eigenspace of the Lichnerowicz Laplacian. For the $SU_n$-type spaces, Gasqui and Goldschmidt found such descriptions in \cite{GGbook} and \cite{GGMem}.  Exploiting these methods leads to the  main result.
\begin{theorem}\label{ThmA}
 The bi-invariant Einstein metric on $SU_{2n+1}$ is isolated in the moduli space of Einstein metrics.
\end{theorem}
\noindent This theorem provides an example of a family of non-product and non-K\"ahler  Einstein metrics admitting infinitesimal deformations, none of which is integrable. In the course of the proof of Theorem \ref{ThmA} we will show that almost all infinitesimal variations of $SU_{2n}$ are obstructed and give a precise characterisation of those which are not. However, we suspect the unobstructed variations are not integrable and that a higher order obstruction will demonstrate this. \\
\\
To prove Theorem \ref{ThmA} we follow the representation theoretic methods used by Koiso \cite{KoiOsaka1}, \cite{KoiOsaka2}, and Gasqui and Goldschmidt \cite{GGMem}. For a symmetric space $G/K$, one can view Koiso's obstruction as an element of the space $\mathrm{Hom}_{G}(s^{2}(\mathfrak{g}),\mathfrak{g})$; an infinitesimal deformation is unobstructed if it is a zero of the associated map. As the relevant Hom-spaces are one or two dimensional, it is possible to argue that the obstruction is a multiple of a particular map.  The theorem follows by showing that the map has no zeros and the obstruction is a non-zero multiple of the map.\\
\\
The strategy employed in the proof is quite general and should be applicable to similar rigidity questions for the symmetric spaces $SU_n/SO_{n}$, $SU_{2n}/Sp_{n}$, and $SU_{p+q}/S(U_{p}\times U_{q})$ which have been open for forty years since Koiso's work. We give a brief discussion of this question and related issues in Section \ref{sec:AOB}. 

\subsection{Rigidity of Ricci solitons}
Einstein metrics can be thought of as fixed points of the Ricci flow
\[
\frac{\partial g}{\partial t} = -2\mathrm{Ric}(g).
\]
Viewed thus, the Einstein metrics are a special case of a type of metric called a Ricci soliton.  One can then ask whether a given Einstein metric $g$ can be deformed through Ricci solitons. The structure theory of the moduli space of Ricci solitons was developed by Podesta and Spiro in \cite{PS} and we refer the reader to Section \ref{sec:solrig} for more details. All Hermitian symmetric spaces admit infinitesimal solitonic deformations.  We are able to prove that for the `odd' Grassmannians, none of these deformations can be integrated to give a Ricci soliton which is not an Einstein metric.
\begin{theorem}\label{ThmB}
For $N=2n+1$, the K\"ahler--Einstein metric on the Grassmannian ${SU_{N}/S(U_{k}\times U_{N-k})}$ is weakly solitonically rigid.
\end{theorem}
\noindent  Our method also yields a new proof of a theorem of Kr\"oncke. This  strengthens the previous result in the case when $k=1$. It  concerns the solitonic rigidity of the Fubini--Study metric on $\mathbb{CP}^{2n}$.
\begin{theorem}[Kr\"oncke, \cite{KKJGA}]\label{ThmC}
The Fubini--Study metric on $\mathbb{CP}^{2n}$ is isolated in the moduli space of Ricci solitons.
\end{theorem}

\subsection{Stability of the Ricci flow}
The relationship between infinitesimal deformations of Einstein metrics, their integrability, and the dynamical stability of the Ricci flow has been known for quite some time through the works of Sesum \cite{Sesum}, Haslhofer and M\"uller \cite{HaMu}, and Kr\"oncke \cite{KKCVP}. In Section \ref{sec:AOB} we compute the third variation of Perelman's $\nu$-functional for infinitesimal deformations of an Einstein metric (this is related to the calculation made for conformal variations by Kr\"oncke \cite{KKCAG} - see also the calculation of Knopf and Sesum \cite{KS}). The next theorem shows that the variation is a multiple of Koiso's obstruction to integrability $\mathcal{I}$ (see Section \ref{sec:deform} for the precise definition of this obstruction).

\begin{theorem}\label{ThmD}
Let $(M^{m},g)$ be an Einstein metric with positive Einstein constant $\frac{1}{2\tau}>0$ and let $h\in s^{2}(T^{\ast}M)$ be an EID. Then if $g(s):=g+sh$,
\[
\frac{d^{3}\nu}{ds^{3}}\bigg|_{s=0}  = \frac{-\tau}{2(4\pi\tau)^{\frac{m}{2}}}\mathcal{I}(h).
\]
\end{theorem}
 
 This yields a new proof of the following result which was first demonstrated by Cao and He in \cite{CH}.
 
\begin{theorem}[Cao--He \cite{CH}]\label{ThmE}
For $n>2$, the bi--invariant metric on $SU_{n}$ is dynamically unstable as a fixed point of the Ricci flow.
\end{theorem}
Note that this result does not depend upon the parity of $n$.  This is because, to show instability, we do not require all infinitesimal Einstein deformations to be obstructed, but just a single one. We can find such a deformation for all $n>2$. 
\subsection{Existing work on rigidity and stability}
As mentioned already, the foundational work on the rigidity of Einstein metrics on symmetric spaces was carried out by Koiso in the papers \cite{KoiOsaka1}, \cite{KoiOsaka2}, and \cite{Koi_EMCS}. In particular, Koiso classified the compact irreducible spaces admitting infinitesimal Einstein deformations. Geometric constuctions of the infinitesimal deformations were given by Gasqui and Goldschmidt in \cite{GGbook} and \cite{GGMem}.  Recently, Derdzinski and Gal \cite{DerdGal} demonstrated that, within the space of left-invariant Einstein metrics on $SU_{n}$, the bi-invariant Einstein metric is isolated.  Theorem \ref{ThmA} removes the restriction that the deformations are through left-invariant metrics in the case when $n$ is odd.\\
\\
The calculation of the second variation of Perelman's $\nu$-entropy and the link between the spectrum of the Lichnerowicz Laplacian and the stability of the Ricci flow appeared in \cite{CHI} (with a detailed proof given in \cite{CZ}).  The dynamical stability of compact symmetric spaces was considered by Cao and He in \cite{CH} where the authors also addressed the stability of compact irreducible symmetric spaces with respect to the Einstein--Hilbert functional (where deformations are through TT tensors). The work of Semmelmann and Weingart \cite{SW} and Schwahn \cite{Sch} now gives a complete understanding of the Einstein--Hilbert picture.\\
\\
There have also been other interesting developments concerning deformability of Einstein metrics and stability questions. For example, there is recent work by Kr\"oncke on sine-cones \cite{KSC}.  The interaction of rigidity and stability questions with other geometric stuctures (e.g. special holonomy) has been taken up by Wang and  Wang in \cite{WW1} and \cite{WW2} as well as the joint works with Semmelmann \cite{SWW2} and \cite{SWW1}.  The relationship between deformability of $G_{2}$ structures and Einstein metrics is investigated by Nagy and Semmelmann in \cite{NS}.  

\subsection{Conventions}
All the manifolds we consider in this paper will be smooth and closed. We follow the convention that the rough Laplacian is  $\Delta h= \mathrm{tr}_{12}(\nabla^{2}h)$ and it therefore has a non-positive spectrum. The curvature tensor is 
\[
\mathrm{R}(X,Y)Z =\nabla^{2}_{X,Y}Z = \nabla_{[X,Y]}Z+\nabla_{Y}\nabla_{X}Z-\nabla_{X}\nabla_{Y}Z.
\]
\\
The curvature operator on symmetric tensors ${\mathrm{Rm}:s^{2}(T^{\ast}M)\rightarrow s^{2}(T^{\ast}M)}$ is defined by
\[
\mathrm{Rm}(h)(X,Y) = \sum_{k}h(R(X,E_{k})Y,E_{k}),
\]
where $\{E_{i}\}$ is a local orthonormal frame. When referring to   symmetric spaces we  adopt the convention of only referring to the manifold with the Riemannian symmetric space metric implicitly understood. Occasionally we will perform tensor calculations using indices and we will use the Einstein summation convention that repeated indices are to be summed over.\\
\\
\textit{Acknowledgements:}\\
\\
We would like to thank Uwe Semmelmann and  Paul-Andi Nagy for their interest in our work, discussions about the paper \cite{PS}, and for bringing the reference \cite{DerdGal} to our attention. We would like to thank Paul Schwahn for bringing a small error in our original proofs of Lemma 4.8 and Lemma 4.9 to our attention.

\section{Deformations of Einstein metrics}\label{sec:deform}
\subsection{General Theory}
We give the essential definitions and theory in this section. Almost all of what is written here is expanded upon and well explained in the book \cite{Bes}.
\begin{definition}[Premoduli space]
Let $(M,g)$ be an Einstein manifold scaled to have unit volume, let $\mathcal{M}_{1}$ be the space of all metrics on $M$ with unit volume, and let $\mathfrak{S}_{g}\subset \mathcal{M}_{1}$ be the slice to the action of the diffeomorphism group $D$ on $\mathcal{M}_{1}$. The subset of Einstein metrics in $\mathfrak{S}_{g}$ is the \textit{premoduli space of Einstein structures about} $g$. We denote this set by $\mathcal{E}(g)$.
\end{definition}

\begin{definition}[Einstein Infinitesimal Deformation (EID)]\label{def:EID}
Let $(M,g)$ be an Einstein manifold and let ${h\in s^{2}(T^{\ast}M)}$  satisfy 
\begin{equation}{\label{h_tf}}
\tr(h)=0,
\end{equation}	
\begin{equation}{\label{h_df}}
\mathrm{div}(h)=0,
\end{equation}
and	
\begin{equation}{\label{h_LL}}
\Delta h +2\mathrm{Rm}(h) =0.
\end{equation}	
The tensor $h$ is referred to as an \textit{essential Einstein infinitesimal deformation} (EID). We denote the space of such $h$ by $\varepsilon(g)$. 
\end{definition}
\noindent As Equation (\ref{h_LL}) is elliptic, we see that the space $\varepsilon(g)$ is always finite dimensional. Tensors that satisfy Equations (\ref{h_tf}) and (\ref{h_df}) are often referred to as \textit{transverse trace-free} or TT tensors. The foundational structure theorem for $\mathcal{E}(g)$ is due to Koiso.
\begin{theorem2}[Koiso \cite{Koi_EMCS}]\label{KST}
Let $(M,g)$ be an Einstein manifold. Then within the slice $\mathfrak{S}_{g}$ there exists a finite dimensional real analytic submanifold $\mathcal{Z}$ such that:
\begin{enumerate}[(i)]
\item The tangent space of $\mathcal{Z}$ at $g$ is the space of essential Einstein infinitesimal deformations $\varepsilon(g)$.
\item The manifold $\mathcal{Z}$ contains the premoduli space $\mathcal{E}(g)$ as a real analytic subset.
\end{enumerate}
\end{theorem2} 
\noindent To try to understand what spaces arise as $\mathcal{E}(g)$, we consider whether the elements of $\varepsilon(g)$ are the genuine tangents to a curve  of metrics in $\mathcal{E}(g)$ passing through the Einstein metric $g$. To make this precise, we introduce the Einstein operator ${E:\mathcal{M}_{1}\rightarrow s^{2}(T^{\ast}M)}$ given by
\begin{equation}\label{Ein_op}
E(g) = \Ric(g)- \left(\frac{\int_{M}\mathrm{S}(g)d\vol_{g}}{\dim(M)}\right) g,
\end{equation}
where $\Ric(g)$ and $\mathrm{S}(g)$ are respectively the Ricci curvature and scalar curvature of $g$. When $\dim(M)>2$, solutions of the equation $E(g)=0$ are Einstein metrics.\\ 
\\
If $g$ is an Einstein metric and $g(t)$ is a curve of metrics in $\mathcal{M}_{1}$  such that $g(0)=g$ and $\dot{g}(0)=h \in \varepsilon(g)$, it is straightforward to check that
$$\frac{d}{dt}E(g(t)) \bigg |_{t=0} = 0.$$
\begin{definition}[Integrability of EIDs]
Let $(M,g)$ be an Einstein manifold and let $k\in \mathbb{N}$. We say that an element $h \in \varepsilon(g)$ is \textit{integrable to order}  $k$ if there exist ${h_{2}, h_{3},\ldots,h_{k}\in s^{2}(T^{\ast}M)}$  such that the curve
$$g_{k}(t):=g+th+\sum_{j=2}^{j=k}\frac{t^{j}}{j!}h_{j},$$
satisfies
$$\frac{d^{j}}{dt^{j}}E(g_{k}(t)) \bigg |_{t=0} = 0,$$
for all $j=1,\ldots,k$. An element $h \in \varepsilon(g)$ is \textit{formally integrable} if there is a formal power series 
$$g(t):=g+th+\sum_{j=2}^{\infty}\frac{t^{j}}{j!}h_{j},$$
such that $E(g(t))=0$.
\end{definition}
\noindent An immediate consequence of Koiso's structure Theorem \ref{KST} is that if ${h\in \varepsilon(g)}$ is formally integrable then there is a smooth curve of Einstein metrics in $\mathcal{E}(g)$ such that $\dot{g}(0)=h$ (in other words a convergent power series solution can be constructed). It is also clear that if $h\in \varepsilon(g) $ is formally integrable, then it must be integrable to order $k$ for all $k\in \mathbb{N}$. Hence it is natural to investigate the obstruction to EIDs being integrable to order two. Formally we see that  
$$\frac{d^{2}}{dt^{2}}E(g(t)) \bigg |_{t=0}  = E''(h,h) +E'(\ddot{g}(0)),$$
where $E'$ and $E''$ are the first and second derivatives (suitably interpreted) of the Einstein operator $E$ defined by (\ref{Ein_op}). Koiso proved a useful way to check integrability to order two.
\begin{lemma}[Koiso, Lemma 4.7 in \cite{KoiOsaka2}]\label{Koilem1}
Let $(M,g)$ be an Einstein manifold. Then $h\in \varepsilon(g)$ is integrable to order two if and only if $E''(h,h) \in \varepsilon(g)^{\perp}$. Here the orthogonal complement is with respect to the $L^{2}$-inner product on $s^{2}(T^{\ast}M)$ induced by $g$.
\end{lemma}
\noindent Using this result, we see that a necessary condition for $h$ to be integrable is the vanishing of the quantity $\langle E''(h,h),h\rangle_{L^{2}}$. This quantity was also computed by Koiso.
\begin{proposition}[Koiso \cite{KoiOsaka2}]
Let $(M,g)$ be an Einstein metric with Einstein constant $\lambda>0$ and let $h \in \varepsilon(g)$. Then an obstruction to the integrability of $h$ to order two is given by the nonvanishing of the quantity
\begin{small}
\begin{equation}\label{Koiobs}
\mathcal{I}(h) :=2\lambda\langle h_{i}^{k}h_{kj},h_{ij}\rangle_{L^{2}}+3\langle\nabla_{i}\nabla_{j}h_{kl},h_{ij}h_{kl} \rangle_{L^{2}}-6\langle\nabla_{i}\nabla_{j}h_{kl},h_{ik}h_{jl}\rangle_{L^{2}},  
\end{equation}
\end{small}
where each of the brackets denotes the $L^{2}$-inner product induced by the metric $g$ on the appropriate bundle.
\end{proposition}
\subsection{Deformations of compact symmetric spaces}
For compact irreducible symmetric spaces, the space of Einstein infinitesimal deformations $\varepsilon(g)$ for the canonical Einstein metrics can be computed using representation theoretic methods.
\begin{theorem2}[Koiso, Theorem 1.1 in \cite{KoiOsaka1}  - see also \cite{GGbook} Proposition 2.40]\label{Infdef_thm}
Let $(M,g)$ be a compact irreducible symmetric space. Then the space of EIDs $\varepsilon(g)$ is $\{0\}$, except in the following cases:
\begin{enumerate}[(i)]
\item $M=SU_n$ with $n\geq 3$, here $\varepsilon(g)\cong \mathfrak{su}(n)\oplus\mathfrak{su}_n$,
\item $M=SU_n/SO_n$ with $n\geq 3$, here $\varepsilon(g)\cong \mathfrak{su}_n$,
\item $M=SU_{2n}/Sp_n$ with $n\geq 3$ here  $\varepsilon(g)\cong \mathfrak{su}_{2n}$,
\item $M=SU_{p+q}/S(U_p \times U_q)$ with $p\geq q\geq2$, here $\varepsilon(g)\cong \mathfrak{su}_{p+q}$,
\item $M=E_{6}/F_{4}$, here $\varepsilon(g)\cong \mathfrak{e}_{6}$.
\end{enumerate} 
\end{theorem2}
\noindent If we write $M=G/K$  where $G$ is the connected component of the identity of the isometry group of $g$, then the previous theorem states that $\varepsilon(g)\cong \mathfrak{g}$; the isomorphisms can be taken to be $G$-equivariant (where $G$ acts on $\mathfrak{g}$ via the adjoint action).  Part of the difficulty in determining whether the infinitesimal deformations are integrable is describing them geometrically, i.e. giving an explicit form of this isomorphism. Fortunately, for spaces {\it (i)-(iii)} in Theorem \ref{Infdef_thm}, there is a straightforward way describing the tensors in $\varepsilon(g)$ due to Gasqui and Goldschmidt.  This  construction is detailed in the memoir \cite{GGMem}; we give explicit differential geometric proofs to show the construction really does produce EIDs.\\
\\
Let $M$ be a symmetric space of compact type; $M$ is thus the coset space of a Riemannian symmetric pair $(G,K)$ where $G$ is a compact, semi-simple Lie group and $K$ is a closed subgroup of $G$. We denote the Lie algebras of $G$ and $K$ by $\mathfrak{g}$ and $\mathfrak{k}$ respectively and note that the inner product on $\mathfrak{g}$ induced by the Killing form yields the $\mathrm{Ad}_{K}$-invariant decomposition 
$$\mathfrak{g} = \mathfrak{k}\oplus \mathfrak{p}.$$
The subspace $\mathfrak{p}$ can be identified with the tangent space of $M$ at the identity coset $eK \in G/K$. For an integer $p\geq 2$ we denote by $s^{p}(\mathfrak{p}^{\ast})^{K}$ the $\mathrm{Ad}_{K}$-invariant symmetric $p$-tensors (or equivalently the $\mathrm{Ad}_{K}$-invariant symmetric degree $p$ polynomials over the space $\mathfrak{p}$). An element $q \in  s^{p}(\mathfrak{p}^{\ast})^{K}$ gives rise to a $G$-invariant symmetric $p$-tensor field $\sigma(q)$ that coincides with $q$ at the identity coset.  Conversely, every such $G$-invariant symmetric tensor field arises this way.\\
\\
We now assume we have a given $q\in s^{3}(\mathfrak{p}^{\ast})^{K}$ and a Killing vector field $\eta \in \mathfrak{iso}(M)$, and define a symmetric $2$-tensor $h_{\eta}$ by
\begin{equation}\label{h_def}
h_{\eta}:= \iota_{\eta}\sigma(q).
\end{equation}
We will henceforth forget the dependence on $q$ and write $\sigma$ for $\sigma(q)$. The following important result is due to Gasqui and Goldschmidt; we give a new and direct Riemannian-geometric proof.
\begin{lemma}[Gasqui--Goldschmidt]
Let $M=G/K$ be a compact irreducible symmetric space  with its canonical Einstein metric $g$. Then the tensor field $h_{\eta}\in \varepsilon(g)$.	
\end{lemma}
\begin{proof}
To show that $h_{\eta}$ satisfies Equation (\ref{h_tf}) we consider the linear map ${\Phi:\mathfrak{p}\rightarrow \mathbb{R}}$ defined by 
$$\Phi(v) :=\tr(\iota_{v}\sigma)|_{(e\cdot K)}.$$ 	As $\Phi$ is $\mathrm{Ad}_{K}$-invariant, a non-trivial kernel would be an $\mathrm{Ad}_{K}$-invariant subspace of $\mathfrak{p}$.  Hence $\Phi$ must vanish identically as the symmetric space $M$ is irreducible. For $a\in G$, the trace of $h_{\eta}$ at a coset $(a\cdot K)$ is ${\Phi(a^{-1}\cdot\eta(a\cdot K)) =0}$ (where we note that the non-vanishing component of $a^{-1}\cdot\eta(a\cdot K)$ is in $\mathfrak{p}$). \\
\\
To demonstrate Equation (\ref{h_df}) we recall the result that any $G$-invariant tensor field on a symmetric space is parallel.  Thus
\[
(\nabla_{\cdot} h_{\eta})(\cdot,\cdot) = \sigma(\nabla_{\cdot} \eta,\cdot,\cdot).
\]
Taking the trace we see that
\[
\mathrm{div}(h_{\eta})(Y) = \sum_{j}\sigma(\nabla_{E_{j}}\eta,E_{j},Y),
\]
for an orthonormal frame $\{E_{j}\}$. However, after using the metric to produce sections of  $\mathrm{End}(TM)$, this is the trace of the composition of the symmetric 2-tensor $\iota_{Y}\sigma$ and the skew-symmetric tensor $\nabla_{\cdot}\eta$ (the skew-symmetry follows from the fact that $\eta$ is a Killing field).  Hence the trace vanishes and $h_{\eta}$ is divergence-free.\\
\\
To demonstrate Equation (\ref{h_LL}) we use the fact that, as the metric is homogeneous, we only need to check the equation
\[
\Delta(h_{\eta})(X,Y)+2\mathrm{Rm}(h_{\eta})(X,Y)=0,
\]
for Killing fields $X,Y$. We also consider a local orthonormal frame $\{E_{i}\}$ with $\nabla E_{i}$ vanishing at the point at which we want to compute the identity. To start with, we note
\[
\Delta h_{\eta}(X,Y) = (\nabla^{2}_{E_{i},E_{i}}h_{\eta})(X,Y) = \sigma(\nabla^{2}_{E_{i},E_{i}}\eta,X,Y),
\]
and
\[
\mathrm{Rm}(h_{\eta})(X,Y) = \sigma(\eta,E_{i},R(X,E_{i})Y)= \sigma(\eta,E_{i},R(Y,E_{i})X).
\]

 In general, given a Killing field $\xi$, we have the identity
\begin{equation}\label{eqn:Dsigma}
\sigma(\nabla_{W}\xi,Y,Z)+\sigma(W,\nabla_{Y}\xi,Z)+\sigma(W,Y,\nabla_{Z}\xi)=0,
\end{equation}
In particular
\begin{equation*}
\sigma(\nabla_{X}\eta,E_{i},Y)+\sigma(\nabla_{E_{i}}\eta,X,Y)+\sigma(\nabla_{Y}\eta,E_{i},X)=0.
\end{equation*}
Taking the derivative with respect to $E_{i}$ and using the fact that we have demonstrated tensors of the form $\iota_{\xi}\sigma$ are divergence free we obtain,
\begin{multline*}
\sigma(\nabla_{E_{i}}\nabla_{X}\eta,E_{i},Y)+\sigma(\nabla_{E_{i}}\nabla_{E_{i}}\eta,X,Y)+\sigma(\nabla_{E_{i}}\eta,\nabla_{E_{i}}X,Y)\\
+\sigma(\nabla_{E_{i}}\eta,X,\nabla_{E_{i}}Y)+\sigma(\nabla_{E_{i}}\nabla_{Y}\eta,E_{i},X)=0.
\end{multline*}
Applying Equation (\ref{eqn:Dsigma}) on the terms $\sigma(\nabla_{E_{i}}\eta,\nabla_{E_{i}}X,Y)$ and $\sigma(\nabla_{E_{i}}\eta,X,\nabla_{E_{i}}Y)$ yields
\[
\sigma(\nabla_{E_{i}}\eta,\nabla_{E_{i}}X,Y) = -\sigma(\nabla_{\nabla_{E_{i}}X}\eta,E_{i},Y).
\]
Hence 
\begin{equation} \label{eqn:Dsigma2}
\sigma(\nabla_{E_{i},X}^{2}\eta,E_{i},Y)+\sigma(\nabla_{E_{i},E_{i}}^{2}\eta, X,Y)+\sigma(\nabla_{E_{i},Y}^{2}\eta,E_{i},X) = 0.
\end{equation}

We now expand
\[
\sigma(\eta,E_{i},\nabla^{2}_{X,E_{i}}Y) = X\cdot \sigma(\eta,E_{i},\nabla_{E_{i}}Y)-\sigma(\nabla_{X}\eta,E_{i},\nabla_{E_{i}}Y)=0
\]
which implies that  Equation (\ref{eqn:Dsigma2}) becomes
\begin{equation}\label{eqn:cyc1}
\sigma(R(X,E_{i})\eta,E_{i},Y)+\sigma(\nabla_{E_{i},E_{i}}^{2}\eta, X,Y)+\sigma(R(Y,E_{i})\eta,E_{i},X)=0.
\end{equation}
We note that we obtain a similar equation simply by swapping $\eta$ and $X$
\begin{equation}\label{eqn:cyc2}
\sigma(R(\eta,E_{i})X,E_{i},Y)+\sigma(\nabla_{E_{i},E_{i}}^{2}X, \eta,Y)+\sigma(R(Y,E_{i})X,E_{i},\eta)=0.
\end{equation}
The next ingredient is an identity (see Lemma 33 in \cite{Petersen} but note the differing sign convention)  which holds for any Killing field $\xi$ on an arbitrary Riemannian manifold:
\[
\nabla_{X,Y}^{2}\xi = R(\xi,X)Y.
\]
Hence 
\[
\sigma(\nabla_{E_{i},E_{i}}^{2}X, \eta,Y)=\sigma(\nabla_{E_{i},E_{i}}^{2}\eta, X,Y) = -\lambda \sigma(\eta,X,Y).
\]
Thus subtracting  Equation (\ref{eqn:cyc1}) from Equation (\ref{eqn:cyc2}) along with an application of the first Bianchi identity yields
\[
\sigma(R(X,\eta)E_{i},E_{i},Y) = \sigma(R(Y,E_{i})\eta,E_{i},X)-\sigma(R(Y,E_{i})X,E_{i},\eta).
\]
The quantity $\sigma(R(X,\eta)E_{i},E_{i},Y)$ vanishes as it is the trace of an antisymmetric map (as in the proof above that $h_{\eta}$ satisifes Equation (\ref{h_df})).  Hence we can reinterpret
Equation (\ref{eqn:cyc1}) as
\[
\sigma(R(Y,E_{i})X,E_{i},\eta)+\sigma(\nabla_{E_{i},E_{i}}^{2}\eta, X,Y)+\sigma(R(X,E_{i})Y,E_{i},\eta) = 0.
\]
This  can be rewritten as 
\[
(\Delta h_{\eta})(X,Y)+2\mathrm{Rm}(h_{\eta})(X,Y) = 0,
\]
as required.
%use again the fact $\sigma$ is parallel which yields
%$$(\nabla^{2}_{Y,Z}h_{\eta})(\cdot,\cdot) = \sigma(\nabla^{2}_{Y,Z}\eta,\cdot,\cdot) =\sigma(\mathrm{Rm}(Y,\eta)Z,\cdot,\cdot).$$
%Hence
%$$\Delta h_{\eta}=\pm\lambda h_{\eta}.$$

\end{proof}

If the space $s^{3}(\mathfrak{p})^{K}$ is non-zero, the map $\eta\rightarrow h_{\eta}$ is injective (any kernel would be a $G$-invariant subspace).  
As $\varepsilon(g)\cong \mathfrak{g}$ this means that all such EIDs arise by this construction. 
\noindent Irreducible symmetric spaces with $G$-invariant symmetric 3-tensors form a very short list.   
\begin{proposition}[Gasqui--Goldschmidt, \cite{GGMem} Proposition 2.1]
Let $M$ be an irreducible simply-connected symmetric space of compact type. The space $s^{3}(\mathfrak{p})^{K}$ vanishes unless $M$ is one of the following spaces:
\begin{enumerate}[(i)]
\item $SU_n$, with $n \geq 3$;
\item $SU_n/SO_n$, with $n \geq 3$;
\item $SU_{2n}/Sp_n$, where $n \geq 3$;
\item $E_{6}/F_{4}$.
\end{enumerate}
For the cases (i)-(iv), then the space $s^{3}(\mathfrak{p})^{K}$ is one-dimensional.
\end{proposition}
\noindent Note the slight difference in the manifolds between  the previous proposition and Theorem \ref{Infdef_thm}.  This is because the infinitesimal deformations of the Grassmannians are not given by $\iota_{v}\sigma$,  since the space $ s^{3}(\mathfrak{p})^{K}$ vanishes for the Grassmannians. We shall discuss this in Section \ref{sec:AOB}.

\section{Input from Representation Theory}\label{sec:reptheory}
\noindent Let $M=G/K$ be a symmetric space of type  {\it (i)-(v)} from Theorem \ref{Infdef_thm}. If we denote by $\psi(h,h)$ the projection of the quantity $E''(h,h)$ to the space of EIDs $\varepsilon (g)$ then $\psi$ can be seen to be a symmetric bilinear $G$-equivariant map. In other words, $\psi \in \mathrm{Hom}_{G}(s^{2}(\mathfrak{g}),\mathfrak{g})$. Lemma \ref{Koilem1} can be restated in the following manner: 
\[
\textrm{an EID } h \in \varepsilon(g) \textrm{ is integrable to order 2 if and only if } \psi(h,h)=0.
\]
The strategy we employ to show that infinitesimal deformations are not integrable can by summarised as follows:
\begin{enumerate} 
\item We identify a one dimensional subspace $\mathcal{V} \subset \mathrm{Hom}_{G}(s^{2}(\mathfrak{g}),\mathfrak{g})$ and show that $\psi \in \mathcal{V}$.  
\item  Compute the exact form of a generator of $\mathcal{V}$, $\Psi$. Hence $\psi =  C\Psi$ for some $C\in\mathbb{R}$. We show that if $\Psi(x,x)=0$ then $x=0$.
\item Find some $h^{\ast}\in \varepsilon(g)$ so that the obstruction in Proposition \ref{Koiobs} \[
\mathcal{I}(h^{\ast}) = \langle \psi(h^{\ast},h^{\ast}),h^{\ast}\rangle_{L^{2}}\neq 0.
\] 
\item Conclude that $C\neq 0$ and so $\psi(h,h)\neq 0$ for any  $h\in \varepsilon(g)$ with $h\neq 0$.
\end{enumerate}\vspace{5pt}
In this paper we will consider this strategy for $G=SU_n\times SU_n$ (and in Section \ref{sec:solrig} we consider $G=SU_n$ for certain infinitesimal solitonic variations). In this section we carry out steps (1) and (2) for these groups.  We expect this material is standard and considerations such as these are implicit in much of the works of Koiso, and Gasqui and Goldschmidt; we include the proofs for completeness. 
\begin{lemma}
\label{lem:dimhom}
Let $G = SU_n \times SU_n$ with $n \ge 2$,  identify ${\mathfrak{g} := \mathrm{Lie}(G)}$ with ${\mathfrak{su}_n \oplus \mathfrak{su}_n}$, and consider $\mathfrak{g}$ as a representation of $G$ via the adjoint representation.  
Then the dimension of real vector space $\mathrm{Hom}_{G}\left( s^{2}( \mathfrak{g}) , \mathfrak{g} \right)$ is given by:
\[
\mathrm{dim} \; \mathrm{Hom}_{G}\left( s^{2}(\mathfrak{g}) , \mathfrak{g} \right)
=
\begin{cases}
0 & \; n = 2 \\
2 & \; n \ge 3.
\end{cases}
\]
\end{lemma}
\begin{proof}
Let $G_{\mathbb C} = SL_n\left(\mathbb C\right) \times SL_n\left(\mathbb C\right)$ and identify $\mathfrak{g}_{\mathbb C} := \mathrm{Lie}\;G_{\mathbb C}$ with $\mathfrak{sl}_n \oplus \mathfrak{sl}_n$.
Let $G_{\mathbb{C}}$ act on $\mathfrak{g}_{\mathbb{C}}$ via the adjoint representation.
It follows from the fact that $G$ is a maximal compact subgroup in $G_{\mathbb{C}}$ that 
\begin{align*}
\mathrm{dim}_{\mathbb{R}} \; \mathrm{Hom}_{G}\left( s^{2}( \mathfrak{g}) , \mathfrak{g} \right) &
 = \mathrm{dim}_{\mathbb{C}} \; \mathrm{Hom}_{G_{\mathbb{C}}} \left( s^{2}( \mathfrak{g}_{\mathbb{C}}) , \mathfrak{g}_{\mathbb{C}} \right). 
\end{align*}
Let $T$ be a maximal torus in $G_{\mathbb{C}}$, $W = N_{G_{\mathbb{C}}} / T$, the associated Weyl group, and $\mathfrak{t} = \mathrm{Lie} \, T$.
The generalised Chevalley restriction theorem proven in Theorem 1 of \cite{Broer} shows that there is an isomorphism of graded vector spaces 
\[
\mathrm{Hom} _{G_{\mathbb{C}}} \left( \mathrm{S}\mathfrak{g}_{\mathbb{C}} , \mathfrak{g}_{\mathbb{C}} \right) 
\cong
\mathrm{Hom} _{W} \left( \mathrm{S}\mathfrak{t} , \mathfrak{t} \right),
\]
where $\mathrm{S} \mathfrak{g}_{\mathbb{C}}$ respectively $\mathrm{S} \mathfrak{t}$ is the symmetric algebra of $\mathfrak{g}_{\mathbb{C}}$ respectively $\mathfrak{t}$, each equipped with the standard grading.
In particular, there is an isomorphism of complex vector spaces
\[
\mathrm{Hom} _{G_{\mathbb{C}}} \left( s^{2}(\mathfrak{g}_{\mathbb{C}}) , \mathfrak{g}_{\mathbb{C}} \right) 
\cong
\mathrm{Hom} _{W} \left( s^{2}(\mathfrak{t}) , \mathfrak{t} \right).
\]

This discussion shows that it suffices to prove that
\begin{equation}
\label{eqn:dimhom}
\mathrm{dim}_{\mathbb{C}} \; \mathrm{Hom}_{W} \left( s^{2}( \mathfrak{t}) , \mathfrak{t} \right)
=
\begin{cases}
0 & \; n = 2 \\
2 & \; n \ge 3 .
\end{cases}
\end{equation}

Fix $T = T' \times T'$ where $T'$ is the standard maximal torus in $SL_n\left(\mathbb{C}\right)$.
Then $W$ is naturally identified with $S_n \times S_n$. Denote by $\mathbb{C}$ the trivial representation of $S_n$, by $U$ the one dimensional alternating representation, by $V$ the standard $n-1$ dimensional irreducible representation, and for $n \ge 3$, by $V_{\left(d-2,2\right)}$ the Specht module corresponding to the partition $\left(d-2,2\right)$.
For $M_1$ and $M_2$ representations of $S_n$ we use $M_1 \boxtimes M_2$ to denote the representation $M_1 \otimes M_2$ of $S_n \times S_n$ in which the first (respectively the second) copy of $S_n$ acts on the first (respectively the second) tensor factor. See for example Chapter 4 of \cite{FultonHarris} for further details about representations of symmetric groups.

\noindent The $S_n \times S_n$ representation $\mathfrak{t}$  decomposes into irreducible $S_n \times S_n$ representations as
\begin{equation}
\label{eqn:tdecomp}
\mathfrak{t} \cong \left( V \boxtimes \mathbb{C} \right) \oplus \left( \mathbb{C} \boxtimes V \right).
\end{equation}
There is the decomposition of the representation $s^{2}(\mathfrak{t})$ as
\begin{align}
\nonumber
s^{2}(\mathfrak{t}) & \cong s^{2} \left( \left( V\boxtimes \mathbb{C} \right) \oplus \left( \mathbb{C} \boxtimes V \right) \right) \\
\nonumber
& \cong s^{2} \left( V \boxtimes \mathbb{C} \right) \oplus \mathrm{s}^2 \left( \mathbb{C} \boxtimes V \right) \oplus \left( V \boxtimes \mathbb{C} \right) \otimes \left( \mathbb{C} \boxtimes V \right) \\
\label{eqn:Stdecomp}
& \cong \left( s^{2}( V) \boxtimes \mathbb{C}\right)  \oplus \left( \mathbb{C} \boxtimes s^{2}(V) \right) \oplus \left( V \boxtimes V \right).
\end{align}
If $d=2$ then $V = U$, $s^{2}(V) = V^{\otimes 2} = \mathbb{C}$, and \eqref{eqn:Stdecomp} decomposes into irreducible $S_n \times S_n$ representations as
\begin{equation}
\label{eqn:Stdecomp1}
\left( \mathbb{C} \boxtimes \mathbb{C} \right) \oplus \left( \mathbb{C} \boxtimes \mathbb{C} \right) \oplus \left( V \boxtimes V \right).
\end{equation}
If $d \ge 3$ then by Exercise 4.19 in \cite{FultonHarris} there is an isomorphism of $S_n$ representations $s^{2}(V) \cong U \oplus V \oplus V_{\left(d-2,2\right)}$, and so \eqref{eqn:Stdecomp} decomposes into irreducible $S_n \times S_n$ representations as
\begin{align}
\nonumber
\left( U \boxtimes \mathbb{C} \right) \oplus \left( V \boxtimes \mathbb{C} \right) \oplus \left( V_{\left(d-2,2\right)} \boxtimes \mathbb{C} \right) \oplus \left( \mathbb{C} \boxtimes U \right) \\
\label{eqn:Stdecomp2}
\oplus \left( \mathbb{C} \boxtimes V \right) \oplus \left( \mathbb{C} \boxtimes V_{\left(d-2,2\right)} \right) \oplus \left( V \boxtimes V \right).
\end{align}
Equation \eqref{eqn:dimhom} then follows from Schur's Lemma and comparing \eqref{eqn:tdecomp} with \eqref{eqn:Stdecomp1} and \eqref{eqn:tdecomp} with \eqref{eqn:Stdecomp2}.
\end{proof}
The space $\mathrm{Hom}_{G}\left( s^{2}( \mathfrak{g}) , \mathfrak{g} \right)$ can be described explicitly.
\begin{lemma} 
\label{lem:mapsHom}
Let $G$ and $\mathfrak{g}$ be as in Lemma \ref{lem:dimhom}. For $n \ge 3$ the vector space $\mathrm{Hom}_{G}\left( s^{2}( \mathfrak{g}) , \mathfrak{g} \right)$ has a basis consisting of the following two functions:
\begin{align*}
\psi_1  \left( \left(X,Y\right) , \left(X',Y'\right) \right) & := \sqrt{-1} \left( \left(XX' + X'X\right) - \frac{1}{n} \mathrm{tr} \left(XX' + X'X\right) \mathrm{Id} , 0 \right), \\
\psi_2 \left( \left(X,Y\right) , \left(X',Y'\right) \right) & := \sqrt{-1}
\left( 0, \left(YY' + Y'Y\right) - \frac{1}{n} \mathrm{tr} \left(YY' + Y'Y\right) \mathrm{Id} \right),
\end{align*}
where $(X,Y)$ and $(X',Y') \in \mathfrak{su}_n\oplus \mathfrak{su}_n$.\\
Furthermore, the subspace  $\mathcal{V}$ of $\mathrm{Hom}_{G}\left( s^{2}( \mathfrak{g}) , \mathfrak{g} \right)$ consisting of maps invariant under the involution on ${\mathfrak{su}_n\oplus \mathfrak{su}_n}$ defined by $(x,y)\rightarrow(y,x)$ is one dimensional and is spanned by the map $\Psi:=\psi_{1}+\psi_{2}$.

\end{lemma}

\begin{proof}
It is straightforward to check that $\psi_1$ and $\psi_2$ are linearly independent symmetric bilinear maps from $\mathfrak{g} \times \mathfrak{g}$ to $\mathfrak{g}$. 
The result then follows from Lemma \ref{lem:dimhom}. The second assertion is easily checked directly.
\end{proof}

\begin{lemma}
\label{lem:zeroset}
Let $\Psi$ be as in Lemma \ref{lem:mapsHom}. Then, for $n$ odd, the zero set of $\Psi$ is $\left\{0\right\}$; for $n$ even, the zero set of $\Psi$ is the $G$-orbit of the plane spanned by the elements $\left(\Lambda,0\right)$ and $\left(0,\Lambda\right)$ of $\mathfrak{g} = \mathfrak{su}_n \oplus \mathfrak{su}_n$, where $\Lambda$ is the block diagonal sum of $\frac{n}{2}$ copies of the matrix
\[
\left(\begin{array}{cc}
\sqrt{-1}& 0 \\
0 & -\sqrt{-1}
\end{array}\right).
\]
\end{lemma}

\begin{proof}
Suppose that $\Psi\left(x\right) = 0$. 
Write $x = \left(X,Y\right)$ where $X,Y \in \mathfrak{su}_n$. 
Then
\[
\Psi\left(x\right) = 2 \sqrt{-1}\; \left(X^2 - \frac{1}{n} \mathrm{tr}(X^2)\mathrm{Id} , Y^2 - \frac{1}{n} \mathrm{tr}(Y^2)\mathrm{Id} \right).
\]
As $\Psi$ is $G$-equivariant we can assume that $X$ and $Y$ are diagonal. 
The kernel of the linear map $T:\mathrm{Mat}_{n}\left(\mathbb{C}\right) \to \mathrm{Mat}_{n}\left(\mathbb{C}\right)$ given by
$$T(X) = X - \frac{1}{n} \mathrm{tr}(X)\mathrm{Id},$$ 
is the set of multiples of the identity matrix $\mathrm{Id}$.
It follows that $\Psi\left(x\right) = 0$ if and only if each of $X$ and $Y$ satisfy the condition that the squares of their eigenvalues are all equal.
For an element $Z$ of $\mathfrak{su}_n$ (the elements of which are traceless with imaginary eigenvalues) this condition is only satisfied if the eigenvalues are either all zero, or if $n$ is even and the eigenvalues are of the form $a\sqrt{-1}$ and $-a\sqrt{-1}$ for some $a \in \mathbb{R}$, with each occurring with multiplicity $\frac{n}{2}$, in which case $Z$ is $SU_n$-conjugate to a multiple of $\Lambda$.
\end{proof}
\noindent Finally in this section, we note that a very similar proof to that in Lemma \ref{lem:dimhom} yields the following result.
\begin{lemma}\label{lem:su_n_hom}
Let $G = SU_n $ with $n \ge 2$,  identify ${\mathfrak{g} := \mathrm{Lie}(G)}$ with ${\mathfrak{su}_n}$, and consider $\mathfrak{g}$ as a representation of $G$ via the adjoint representation.  
Then the dimension of real vector space $\mathrm{Hom}_{G}\left( s^{2}( \mathfrak{g}) , \mathfrak{g} \right)$ is given by:
\[
\mathrm{dim} \; \mathrm{Hom}_{G}\left( s^{2}(\mathfrak{g}) , \mathfrak{g} \right)
=
\begin{cases}
0 & \; n = 2 \\
1 & \; n \ge 3.
\end{cases}
\]
\end{lemma}

\section{Calculations for $SU_n$}\label{sec:SU(n)calcs}
\noindent In this section we give the proof of the Theorem \ref{ThmA} using the strategy outlined in the previous section; here we focus on step (3) - demonstrating that the obstruction map does not vanish identically.  It is important to fix conventions at this stage to ensure the correct calculation of various constants. The Lie algebra $\mathfrak{su}_n$ is the space of skew-Hermitian $n\times n$ matrices.  We endow $\mathfrak{su}_n$ with the Euclidean inner product 
$$\langle A,B \rangle:= -\tr(AB).$$
This inner product is $(1/2n)$ times the one induced by the Killing form and so the resulting bi-invariant Einstein metric has Einstein constant $\lambda=n/2$ (as the Killing form induces a metric with Einstein constant $1/4$).\\
\\
\noindent To construct the EIDs given by the Gasqui--Goldschmidt construction (\ref{h_def}) we require a tensor $\sigma\in s^{3}(\mathfrak{su}(n))^{SU_n}$.  Up to scale, there is only one such tensor given by
\begin{equation}\label{3tens_su(n)}
\sigma(X,Y,Z) := \sqrt{-1}\left(\tr(XYZ)+\tr(XZY)\right), 
\end{equation}
where $X,Y, Z\in \mathfrak{su}_n.$\\
\\
The following result allows us to reduce computing Koiso's obstruction (\ref{Koiobs}) to a pointwise calculation (rather than having to integrate over the whole manifold).
\begin{lemma}\label{LeftinvLem}
Let $\eta \in \mathfrak{su}_n$ and denote by $\tilde{\eta}$ the corresponding left-invariant (resp. right-invariant) extension to a vector field on $SU_n$.  Then the tensors $h_{\eta}, \nabla h_{\eta}$ and $\nabla^{2}h_{\eta} $ are left-invariant  (resp. right-invariant).
\end{lemma}
\begin{proof}
We shall demonstrate the left-invariant statement. Let $g,k\in SU_n$ and let $X,Y \in  \mathfrak{X}(SU_n)$. We compute
$$L_{g}^{\ast}(h_{\eta})(X,Y)(k) = h_{\eta}(L_{g\ast}X,L_{g\ast}Y)(gk) = \sigma(\tilde{\eta}(gk),L_{g\ast}X,L_{g\ast}Y)(gk),$$ 
and thus, as $\tilde{\eta}$ is left-invariant, we obtain
$$L_{g}^{\ast}(h_{\eta})(X,Y)(k) = L_{g}^{\ast}(\sigma)(\tilde{\eta},X,Y)(k) =\sigma(\tilde{\eta},X,Y)(k) =  h_{\eta}(X,Y)(k). $$
The result follows by noting that left or right invariance is preserved by taking covariant derivatives. 
\end{proof}

A corollary of Lemma \ref{LeftinvLem} is that the map $\psi$ must lie in the one dimensional subspace of maps that are invariant under the involution swapping the two factors of $\mathfrak{g}=\mathfrak{su}_{n}\oplus\mathfrak{su}_{n}$. Each factor corresponds to the extension of $\mathfrak{su}_{n}$ as left-invariant or right-invariant vector fields and Lemma \ref{LeftinvLem} demonstrates that the obstruction map $\psi$ must take the same value on either extension.  This means the obstruction is a multiple of the map $\Psi$ described in Lemma \ref{lem:mapsHom}.

\subsection{Results in $\mathfrak{su}_n$}

If we use the previous lemma to produce a left-invariant tensor $h_{\eta}$ and choose the identity as the point in $SU_n$  at which to perform calculations, the algebra we need to do takes place in $\mathfrak{su}(n)$. In this subsection we collect the relevant results about the Euclidean vector space ${(\mathfrak{su}_n, \langle \cdot , \cdot \rangle)}$.  All of the following lemmata are proved directly using straightforward (if tedious) matrix algebra; most of the proofs are omitted.\\ 
\\
In order to perform calculations, we fix an orthonormal basis of $\mathfrak{su}_n$.  \begin{lemma}\label{ONBofsu_n} 
For $1\leq k\leq n-1$, define matrices
$$
T_{k} = \frac{\sqrt{-1}}{\sqrt{k(k+1)}}\mathrm{Diag}\left(\underbrace{1,1,\dots,1}_{k \ \mathrm{terms}},-k,0,\dots,0\right),
$$ 
and,
for $1\leq k<l \leq n$, define matrices $E^{r}(k,l)$ and $E^{c}(k,l)$  by
$$(E^{r}(k,l))_{pq} = \frac{1}{\sqrt{2}}\left(\delta_{kp}\delta_{lq}-\delta_{kq}\delta_{lp} \right),$$ 
and
$$(E^{c}(k,l))_{pq} = \frac{\sqrt{-1}}{\sqrt{2}}\left(\delta_{kp}\delta_{lq}+\delta_{kq}\delta_{lp} \right).
$$
Then the matrices 
$$
\mathcal{B} = \{T_{1},T_{2},\dots,T_{n-1}\}\cup \{E^{r}(k,l),E^{c}(k,l) \}_{1\leq k<l \leq n},
$$
form an orthonormal basis of $(\mathfrak{su}(n), \langle \cdot, \cdot \rangle)$ with the $T_{k}$ spanning to the Lie subalgebra of a maximal torus generated by the diagonal matrices in $SU_n$.
\end{lemma}
We see from the form of the tensor $\sigma$ in Equation (\ref{3tens_su(n)}), it will be useful to consider various quantities of the form $XY+YX$ with $X,Y\in \mathfrak{su}(n)$. To do this, we introduce a small variation on the notation for the basis given in Lemma \ref{ONBofsu_n}. Given an unordered pair $\{k,l\}$, if $a=\min\{k,l\}$ and $b=\max\{k,l\}$, then we define 
$$E^{r}(\{ k,l\}) = E^{r}(a,b) \qquad  \mathrm{and} \qquad E^{c}(\{ k,l\}) = E^{c}(a,b).$$

\begin{lemma}\label{SquareLem}
Let $T_{k}$, $E^{r}(k,l)$, and $E^{c}(k,l)$ be the orthonormal basis of $(\mathfrak{su}_n,\langle \cdot,\cdot \rangle)$ defined in Lemma \ref{ONBofsu_n}. Then %\edz{Does changing the numbering here matter? Wanted to sort out bloody margins. Please check against old version!!}
\begin{enumerate}[(i)]
\item   If $p=q$, then
$$
T_{p}T_{q}+T_{q}T_{p} = \dfrac{2}{p(p+1)}\mathrm{Diag}\left(\underbrace{-1,-1,\dots, -1}_{p \ \mathrm{terms}},-p^{2},0,\dots,0 \right) 
$$
\item If $p>q$, then 
$$
T_{p}T_{q}+T_{q}T_{p} = \dfrac{2}{\sqrt{pq(p+1)(q+1)}}\mathrm{Diag}\left(\underbrace{-1,-1,\dots, -1}_{q\ \mathrm{terms}},q,0,\dots,0 \right)$$
\item $T_{p}E^{r}(k,l)+E^{r}(k,l)T_{p} = ((T_p)_{kk}+(T_{p})_{ll})E^{r}(k,l)$
\item $T_{p}E^{c}(k,l)+E^{c}(k,l)T_{p} =  ((T_p)_{kk}+(T_{p})_{ll})E^{c}(k,l) $
\item  If $\{k,l\}\neq\{p,q\}$, then $E^{r}(k,l)E^{r}(p,q)+E^{r}(p,q)E^{r}(k,l)  =$
$$ 
 \dfrac{\sqrt{-1}}{\sqrt{2}}\left(\delta_{ql}E^{c}(\{p,k\}) 
 -\delta_{qk}E^{c}(p,l)-\delta_{lp}E^{c}(k,q)+ \delta_{pk}E^{c}(\{ l,q\})\right).
 $$
 
\item  If $k=p$  and  $l=q$, then  $E^{r}(k,l)E^{r}(p,q)+E^{r}(p,q)E^{r}(k,l)  =$
$$
\mathrm{Diag}\left(\underbrace{0,0,\dots,0}_{k-1 \ \mathrm{terms}},-1,\underbrace{0,0,\dots, 0}_{l-k-1 \ \mathrm{terms}},-1,0,0,\dots ,0\right) 
$$

\item  If $\{k,l\}\neq\{p,q\}$, $E^{r}(k,l)E^{c}(p,q)+E^{c}(p,q)E^{r}(k,l) = $
$$
\dfrac{\sqrt{-1}}{\sqrt{2}}\left(\pm\delta_{ql}E^{r}(\{p,k\})+\delta_{pl}E^{r}(k,q)+\delta_{kq}E^{r}(p,l)\pm\delta_{pk}E^{r}(\{q,l\}) \right).
$$
\item If $  k=p \ \mathrm{and} \ l=q$, then $E^{r}(k,l)E^{c}(p,q)+E^{c}(p,q)E^{r}(k,l) =0$.

\item If $\{k,l\}\neq\{p,q\},$, then $E^{c}(k,l)E^{c}(p,q)+E^{c}(p,q)E^{c}(k,l) =$
$$\dfrac{\sqrt{-1}}{\sqrt{2}}\left(\delta_{ql}E^{c}(\{p,k\})+\delta_{pl}E^{c}(k,q)+\delta_{kq}E^{c}(p,l)+\delta_{pk}E^{c}(\{q,l\})\right)
$$
\item If $k=p$ and $l=q$, then $E^{c}(k,l)E^{c}(p,q)+E^{c}(p,q)E^{c}(k,l) =$
$$
\mathrm{Diag}\left(\underbrace{0,0,\dots,0}_{k-1 \ \mathrm{tems}},-1,\underbrace{0,0,\dots, 0}_{l-k-1 \ \mathrm{terms}},-1,0,0,\dots 0\right) 
 $$
\end{enumerate}
\end{lemma}
We will see that the precise signs of the terms in (vii) will not matter; the real utility of Lemma \ref{SquareLem} is that if we choose basis elements $X,Y, \in \mathcal{B}$ with $X\neq Y$, then 
$$\sqrt{-1}(XY+YX) \perp T_{n-1}.$$
Thus choosing $\eta$ to be a convenient multiple of $T_{n-1}$ turns out to greatly simplify calculations. 
 
\begin{lemma}\label{HdiagLem}
	Let $\eta = \sqrt{-1}\mathrm{Diag}(\underbrace{1,1,\dots,1}_{n-1},-(n-1)) $ and let  $H\in s^{2}(\mathfrak{g}^{\ast})$ be defined by
	$$H(X,Y) = \sqrt{-1}\hspace{1pt}\mathrm{tr}(\eta(XY+YX)).$$
Then $H$ is diagonal with respect to the orthonormal basis $\mathcal{B}$ defined in Lemma \ref{ONBofsu_n}.  Furthermore	
\begin{align} 
H(T_{p},T_{p}) = 2 \qquad \mathrm{for} \quad p<n-1, \label{H_diag_eqns1}\\
H(T_{n-1},T_{n-1})  = -2(n-2),\\
H(E^{r}(k,l),E^{r}(k,l))=H(E^{c}(k,l),E^{c}(k,l))=2 \qquad \mathrm{for} \quad l<n,\label{H_diag_eqns3}\\
H(E^{r}(k,n),E^{r}(k,n))=H(E^{c}(k,n),E^{c}(k,n))=-(n-2)\label{H_diag_eqns4}.
\end{align}	
\end{lemma}

\begin{proof}
We begin by noting that $\eta = \sqrt{n(n-1)}T_{n-1}$ and 
$$H(X,Y) = \langle \eta, \sqrt{-1}(XY+YX) \rangle.$$
Using Lemma \ref{SquareLem} we see that all the non-diagonal terms of $H$ vanish.  The remaining diagonal terms are easily calculated.
\end{proof}

\noindent In order to compute the terms of Koiso's obstruction (\ref{Koiobs}) that contain covariant derivatives we will need information about Lie brackets $[\cdot ,\eta]$. The following is again straightforward algebra.

\begin{lemma}\label{LieBracketLem1}
Let $\eta$ be as in Lemma \ref{HdiagLem}. Then for any ${X\in \mathrm{Mat}(\mathbb{C})^{n\times n}}$ we have
$$[X,\eta] = n\sqrt{-1}\left(
\begin{array}{ccccc}
0 & 0 & \dots & 0 & -X_{1n}\\
0 & 0 & \dots & 0 & -X_{2n}\\
\vdots & \vdots & \ddots & \vdots &\vdots\\
0 & 0 & \dots &0 & -X_{(n-1)n}\\
X_{n1} &X_{n2} & \dots &X_{n(n-1)} & 0 
\end{array}
\right).$$
\end{lemma}
The fact that $H$ is diagonal will mean we need only compute very few Lie-bracket-type terms.  The following collects what we will need.

\begin{lemma}\label{LieBracketLem2}
Let $\eta$ be as in Lemma \ref{HdiagLem}. Then
\begin{enumerate}[(i)]
\item for $0<k<l<n$ 
$$[E^{r}(k,l),[E^{r}(k,l),\eta]]=[E^{c}(k,l),[E^{c}(k,l),\eta]]=0,$$
\item $[E^{r}(k,n),[E^{r}(k,n),\eta]]=n\sqrt{-1}\mathrm{Diag}\left(\underbrace{0,0,\dots,0}_{k-1 \ \mathrm{terms}},-1,0,0,\dots,0,1\right),$
\item $[E^{c}(k,n),[E^{c}(k,n),\eta]]=n\sqrt{-1}\mathrm{Diag}\left(\underbrace{0,0,\dots,0}_{k-1 \ \mathrm{terms}},-1,0,0,\dots,0,1\right).$
%\item $E^{r}(k,l)^{2} = E^{c}(k,l)^{2} = -\dfrac{1}{2}\mathrm{Diag}(\underbrace{0,0,\dots,0}_{k-1},1,\underbrace{0,0,\dots, 0}_{l-k-1},1,0,0,\dots) $
%\item $T_{p}^{2} =\dfrac{1}{p+1}\mathrm{Diag}(\underbrace{-1,-1,\dots, -1}_{p+1 \ \mathrm{terms}},0,\dots,0)$
\end{enumerate}
\end{lemma}
\begin{proof}
Part (i) follows immediately from Lemma \ref{LieBracketLem1} as, for any $0\leq j \leq n$, $${(E^{r}(k,l))_{jn}=(E^{c}(k,l))_{jn}=0},$$ if $l<n$. Parts (ii) and (iii) are also straightforward after using Lemma \ref{LieBracketLem1} to compute the first Lie bracket.
\end{proof}
\subsection{Computing $\mathcal{I}$} In this section we consider the symmetric 2-tensor $h_{\eta}$ that is generated by the inclusion of the left-invariant extension of $\eta$ in $\sigma$ defined by Equation (\ref{3tens_su(n)}). We proceed by computing each of the terms of the obstruction $\mathcal{I}$ given in Equation (\ref{Koiobs}).
\begin{lemma}\label{h_etaLem_1}
Let $\eta$ be as in Lemma \ref{HdiagLem} and let $h_{\eta}=\iota_{\tilde{\eta}}\sigma$, where $\tilde{\eta}$ is the left-invariant extension of $\eta$ and $\sigma$ is the symmetric 3-tensor defined in Equation (\ref{3tens_su(n)}). Then
$$\langle (h_{\eta})^{k}_{i}(h_{\eta})_{kj},(h_{\eta})_{ij} \rangle_{L^{2}} =2(n-1)(n-2)(12-n^{2}) \vol(SU_n),$$
 where $\vol(SU_n)>0$ is the volume of the bi--invariant metric on $SU_n$ induced the inner product $\langle \cdot,\cdot \rangle$ on $\mathfrak{su}(n)$. 
 \end{lemma}
\begin{proof}
The tensor $h_{\eta}$ is left-invariant by Lemma \ref{LeftinvLem}.  Hence we can compute at the identity and so 
$$
\langle (h_{\eta})^{k}_{i}(h_{\eta})_{kj},(h_{\eta})_{ij} \rangle_{L^{2}} = \langle H_{i}^{k}H_{kj},H_{ij}\rangle \vol(SU_n),
$$
where $H$ is the symmetric bilinear form defined in Lemma \ref{HdiagLem} (we also use $\langle \cdot, \cdot \rangle$ to denote the inner product on $s^{2}(\mathfrak{su}(n))$ induced by the one on $\mathfrak{su}(n)$). As $H$ is diagonal with respect to the orthonormal basis $\mathcal{B}$ defined in Lemma \ref{ONBofsu_n} we can compute
$$\langle H_{i}^{k}H_{kj},H_{ij}\rangle = \mathrm{tr} (H^{3}), $$
where the trace on the right-hand side is the usual trace of a matrix. Using equations (\ref{H_diag_eqns1})-(\ref{H_diag_eqns4}) we can compute
$$\mathrm{tr} (H^{3})= ((n-2)\times 8)-(8(n-2)^{3}) + ((n-2)(n-1)\times 8)-((n-1)\times(n-2)^3) .$$
The result follows after simplification.
\end{proof}

\begin{lemma}\label{h_etaLem_2}
Let $h_{\eta}$ be as in Lemma \ref{h_etaLem_1}. Then  
$${\langle\nabla_{i}\nabla_{j}(h_{\eta})_{kl},(h_{\eta})_{ij}(h_{\eta})_{kl}\rangle_{L^{2}}  =[(n - 2)^{2}(n-1)n(n+2)]\vol(SU_n).
}
$$
\end{lemma}
\begin{proof}As  $\sigma$ is parallel, we have (computing at the identity) 
$$\nabla_{i}\nabla_{j}(h_{\eta})_{kl} = \sigma(\nabla_{i}\nabla_{j}\tilde{\eta}-\nabla_{\nabla_{E_{i}}E_{j}}\tilde{\eta},E_{k},E_{l}),$$
where $\{E_{i}\}$ is any basis of $\mathfrak{g}$. We are using bi-invariant metric on $SU_n$ so, using left-invariant extensions of the basis $\{E_{i}\}$, we have 
\[
\nabla_{\nabla_{E_{i}}E_{j}}\tilde{\eta} = \frac{1}{4}\sigma([[E_{i},E_{j}],\eta],E_{k},E_{l}),
\]
and
$$\nabla_{i}\nabla_{j}\tilde{\eta} = \frac{1}{4}[E_{i},[E_{j},\eta]].$$
As $\nabla_{\nabla_{E_{i}}E_{j}}\tilde{\eta}$ is antisymmetric in the indices $i$ and $j$ we have
\[
\langle\sigma(\nabla_{\nabla_{E_{i}}E_{j}}\tilde{\eta},E_{k},E_{l}), h_{ij}h_{kl}\rangle_{L^{2}} = 0, 
\]
and thus 
$$
\langle\nabla_{i}\nabla_{j}(h_{\eta})_{kl},(h_{\eta})_{ij}(h_{\eta})_{kl}\rangle_{L^{2}} = \frac{1}{4}\sigma([E_{i},[E_{j},\eta]],E_{k},E_{l}) H_{ij}H_{kl} \vol(SU_n).
$$
If we take $\{E_{i}\} = \mathcal{B}$, and as $H$ is diagonal with respect to the basis  $\mathcal{B}$, we need only consider $E_{i} =E_{j}$ and $E_{k}=E_{l}$ and so  
$$
\langle\nabla_{i}\nabla_{j}(h_{\eta})_{kl},(h_{\eta})_{ij}(h_{\eta})_{kl}\rangle_{L^{2}} = \frac{1}{4}\sigma([E_{i},[E_{i},\eta]],E_{k},E_{k}) H_{ii}H_{kk}.
$$
From Lemma \ref{LieBracketLem2} we see the right-hand side of the previous equation is non-zero only when $E_{i}=E^{r}(k,n)$ or $E_{i} = E^{c}(k,n)$; in either case, ${H_{ii}= -2(n-2)}$ and we  obtain using Lemma \ref{LieBracketLem2}
$$ \frac{1}{4}\langle [E_{i},[E_{j},\eta]] , H_{ij}H_{kl} \rangle = -\left(\frac{n(n-2)}{2}\right)\sum_{k=1}^{n-1}\langle \sigma(v_{k},\cdot,\cdot),H(\cdot,\cdot)\rangle,$$
where 
$$v_{k} =\sqrt{-1}\mathrm{Diag}\left(\underbrace{0,0,\dots,0}_{k-1 \ \mathrm{terms}},-1,0,0,\dots,0,1\right).$$
We split the sum up as follows:
\begin{align*}
\langle \sigma(v_{k},\cdot,\cdot),H(\cdot,\cdot)\rangle  &=  \sum_{i=1}^{n-1}\sigma(v_{k},T_{i},T_{i})H(T_{i},T_{i})\\
 &+ 2\sum_{1\leq i<j<n}\sigma(v_{k},E^{r}(i,j),E^{r}(i,j))\\
 &+ 2\sum_{1\leq i<j<n}\sigma(v_{k},E^{c}(i,j),E^{c}(i,j))\\
&-(n-2) \sum_{i=1}^{n-1}\sigma(v_{k},E^{r}(i,n),E^{r}(i,n))\\
&-(n-2) \sum_{i=1}^{n-1}\sigma(v_{k},E^{c}(i,n),E^{c}(i,n)),
\end{align*}
where we have used Equations (\ref{H_diag_eqns3}) and (\ref{H_diag_eqns4}) to simplify the final two summations. 
Using Lemma \ref{SquareLem} and the definition of $\sigma$ we can compute 
$$
\langle \sigma(v_{k},\cdot,\cdot),H(\cdot,\cdot)\rangle = \sum_{i=1}^{n-1}\sigma(v_{k},T_{i},T_{i})H(T_{i},T_{i}) \ -4(n-2) -2(n-2)^{2}.
$$
Summing over the possible values of $k$ yields
\begin{align*}
\sum_{k=1}^{n-1}\langle \sigma(v_{k},\cdot,\cdot),H(\cdot,\cdot)\rangle &= \left(\sum_{k=1}^{n-1}\sum_{i=1}^{n-1}\sigma(v_{k},T_{i},T_{i})H(T_{i},T_{i})\right)\\
&-4(n-1)(n-2) -2(n-1)(n-2)^{2}.
\end{align*}
Switching the order of summation we can compute
$$
\sum_{k=1}^{n-1}\sigma(v_{k},T_{i},T_{i})H(T_{i},T_{i}) = \left\{\begin{array}{cc}-4 & \mathrm{if} \ i<n-1,\\
-4(n-2)^{2} & \mathrm{if}  \ i=n-1. 
\end{array}\right.
$$
Hence 
\begin{align*}
\sum_{k=1}^{n-1}\langle \sigma(v_{k},\cdot,\cdot),H(\cdot,\cdot)\rangle &= -4(n-2)(n-1)-4(n-1)(n-2) -2(n-1)(n-2)^{2}\\
&= -2(n-2)(n-1)(n+2).
\end{align*}
The result now follows.
\end{proof}

\begin{lemma}\label{h_etaLem_3}
Let $h_{\eta}$ be as in Lemma \ref{h_etaLem_1}. Then
$$
\langle\nabla_{i}\nabla_{j}h_{kl},h_{ik}h_{jl}\rangle_{L^{2}} =-\dfrac{1}{2}[(n - 2)^{2}(n-1)n(n+2)]\vol(SU_n). 
$$
\end{lemma}
\begin{proof}
As  in the proof of Lemma \ref{h_etaLem_2},  the fact that $\sigma$ is parallel means that, for any local frame $\{E_{i}\}$,
$$\nabla_{i}\nabla_{j}h_{kl} = \sigma(\nabla_{i}\nabla_{j}\tilde{\eta}-\nabla_{\nabla_{E_{i}}E_{j}}\tilde{\eta}, E_{k},E_{l}).$$
We note that, at the identity,
\begin{align*}
\sigma(\nabla_{\nabla_{E_{i}}E_{j}}\tilde{\eta},E_{k},E_{l}) &= \frac{1}{4}\sigma([[E_{i},E_{j}],\eta],E_{k},E_{l}) \\
 &=-\frac{1}{4}\left( \sigma(\eta,[[E_{i},E_{j}],E_{k}],E_{l})+\sigma(\eta,E_{k},[[E_{i},E_{j}],E_{l}])\right),\\
\end{align*}
where the final equality follows from the $\textrm{Ad}$-invariance of $\sigma$.  The Jacobi identity then yields 
\begin{align*}
\sigma(\nabla_{\nabla_{E_{i}}E_{j}}\tilde{\eta},E_{k},E_{l}) &= 
\frac{1}{4}\left(\sigma(\eta,[[E_{k},E_{i}],E_{j}],E_{l})+\sigma(\eta,[[E_{j},E_{k}],E_{i}],E_{l})\right.\\
&+\left.\sigma(\eta,E_{k},[[E_{l},E_{i}],E_{j}])+\sigma(\eta,E_{k},[[E_{j},E_{l}],E_{i}])\right).\\
\end{align*}
Thus $\langle \sigma(\nabla_{\nabla_{E_{i}}E_{j}}\tilde{\eta},E_{k},E_{l}),(h_{\eta})_{ik}(h_{\eta})_{jl} \rangle_{L^{2}}$ can be expanded as 
\begin{align*}
& \ \ \  \frac{1}{4}\langle \sigma(\eta,[[E_{j},E_{k}],E_{i}],E_{l})+\sigma(\eta,E_{k},[[E_{l},E_{i}],E_{j}]),(h_{\eta})_{ik}(h_{\eta})_{jl} \rangle_{L^{2}}\\
&= \frac{1}{4}\langle \sigma(\eta,[[E_{j},E_{k}],E_{i}],E_{l})+\sigma(\eta,[[E_{l},E_{i}],E_{j}],E_{k}),(h_{\eta})_{ik}(h_{\eta})_{jl} \rangle_{L^{2}}\\
&=\frac{1}{4}\langle \sigma(\eta,[[E_{j},E_{k}],E_{i}],E_{l})+\sigma(\eta,[[E_{k},E_{j}],E_{i}],E_{l}),(h_{\eta})_{ik}(h_{\eta})_{jl} \rangle_{L^{2}}\\
&= 0,
\end{align*}
where the penultimate equality follows by a reordering of indices (note we are summing over all possible $i,j,k,$ and $l$).\\
\\
Computing at the identity in the orthonormal basis $\mathcal{B}$, Lemma \ref{HdiagLem} yields  $h_{\eta}=H$ is diagonal and so, to compute ${\langle\nabla_{i}\nabla_{j}h_{kl},h_{ik}h_{jl}\rangle }$, we need only consider terms of the form
$$\sigma(\nabla_{i}\nabla_{j}\eta, E_{i},E_{j}) = \dfrac{1}{4}\sigma([E_{i},[E_{j},\eta]], E_{i},E_{j}).$$
Noting that $\sigma$ is Ad-invariant  and  we can write
$$\sigma(\nabla_{i}\nabla_{j}\eta, E_{i},E_{j}) = -\dfrac{1}{4}\sigma([E_{j},\eta],E_{i},[E_{i},E_{j}]).$$
Using the definition of $\sigma$ given by Equation(\ref{3tens_su(n)}) we have 
\begin{align*}
%-\dfrac{1}{4}\sigma([E_{j},\eta],E_{i},[E_{i},E_{j}]) & = -\dfrac{\sqrt{-1}}{4}\mathrm{tr}([E_{j},\eta](E_{i}[E_{i},E_{j}]+[E_{i},E_{j}]E_{i}))\\
-\dfrac{1}{4}\sigma([E_{j},\eta],E_{i},[E_{i},E_{j}]) &  =  -\dfrac{\sqrt{-1}}{4}\mathrm{tr}([E_{j},\eta](E_{i}[E_{i},E_{j}]+[E_{i},E_{j}]E_{i})) \\
& =   -\dfrac{\sqrt{-1}}{4}\mathrm{tr}([E_{j},\eta](E_{i}^{2}E_{j}-E_{j}E_{i}^{2}))\\
 &=-\dfrac{\sqrt{-1}}{4}\langle [E_{j},\eta],[E_{j},E_{i}^{2}]\rangle \\
& = \dfrac{\sqrt{-1}}{4}\langle [E_{j},[E_{j},\eta]],E_{i}^{2}\rangle,
\end{align*}
where the final equality follows as the inner product $\langle \cdot ,\cdot\rangle$ is also Ad-invariant. Thus we have (relabelling indices as needed)
\begin{align*}
\langle\nabla_{i}\nabla_{j}h_{kl},h_{ik}h_{jl}\rangle_{L^{2}} &= -\dfrac{\sqrt{-1}}{4}\mathrm{tr}([E_{i},[E_{i},\eta]]E_{k}^{2})H_{ii}H_{kk}\vol(SU_n) \\
&= -\frac{1}{8}\sigma([E_{i},[E_{i},\eta]],E_{k},E_{k})H_{ii}H_{kk}\vol(SU_n).\\
&= -\frac{1}{2}{\langle\nabla_{i}\nabla_{j}(h_{\eta})_{kl},(h_{\eta})_{ij}(h_{\eta})_{kl}\rangle_{L^{2}}},\\
& = -\dfrac{1}{2}[(n - 2)^{2}(n-1)n(n+2)]\vol(SU_n),
\end{align*}
where  Lemma \ref{h_etaLem_2} was applied to deduce the final equality. 
\end{proof}
\noindent We now  have everything needed to prove the main theorem in this paper.

\begin{proof}[Proof of Theorem \ref{ThmA}] As $\psi \in \mathrm{Hom}_{G}(s^{2}(\mathfrak{g}), \mathfrak{g})$ and it is invariant under the involution switching the $\mathfrak{su}_{n}$ factors, it must be a multiple of $\Psi$. As $2n+1$ is odd,  the obstruction to integrability to second order $\psi$ has no non-trivial zeroes by Lemma \ref{lem:zeroset} (unless of course it vanishes identically).  Thus, to show that none of the elements of $\varepsilon(g)$ are integrable to second order we simply need to find $h^{\ast}\in \varepsilon(g)$ such that $\mathcal{I}(h^{\ast})\neq 0$. We take $h^{\ast} = h_{\eta}$ where $\eta$ is as in Lemma \ref{HdiagLem},  and then, using Lemmas \ref{h_etaLem_1}, \ref{h_etaLem_2}, and \ref{h_etaLem_3}, and the fact that the Einstein constant $\lambda = \dfrac{n}{2}$ we compute

\begin{align*}
\mathcal{I}(h^{\ast}) =& 2\lambda\langle h_{i}^{k}h_{kj},h_{ij}\rangle+3\langle\nabla_{i}\nabla_{j}h_{kl},h_{ij}h_{kl} \rangle- 6\langle\nabla_{i}\nabla_{j}h_{kl},h_{ik}h_{jl}\rangle \\
=& \left(2\cdot \frac{n}{2} \left[2(n-1)(n-2)(12-n^{2})\right]\right. \\
&\left. +3\left[(n - 2)^{2}(n-1)n(n+2) \right]\right. \\
&\left. +6\dfrac{1}{2}[(n - 2)^{2}(n-1)n(n+2)] \right) \vol(SU_n).
\end{align*}
Simplifying, this yields
\[
\mathcal{I}(h^{\ast}) = 4n^{3}(n-1)(n-2)\vol(SU_n) \neq 0.
\]
Hence all the infinitesimal variations are non-integrable and so the Einstein metric is isolated in its moduli space.
\end{proof}
\section{Applications to Solitonic Rigidity}\label{sec:solrig}
\subsection{Ricci Solitons}
Einstein metrics are fixed points (up to diffeomorphism and homothethic scaling) of the Ricci flow
\begin{equation}\label{RFequation}
\frac{\partial g}{\partial t} = -2 \Ric (g).
\end{equation}
There are also non-Einstein fixed points of Equation \ref{RFequation} known as \textit{Ricci solitons} which are metrics that solve
\begin{equation}\label{RSequation}
\Ric (g) +\frac{1}{2}L_{X}g =\lambda g,
\end{equation}
for a vector field $X\in \Gamma (TM)$ and constant $\lambda\in \mathbb{R}$.  A soliton is known as \textit{expanding}, resp. \textit{steady}, resp. \textit{shrinking} depending upon whether $\lambda<0$, resp. $\lambda=0,$ resp. $\lambda>0$. If one can find a function $f$ such that $X=\nabla f$ then Equation (\ref{RSequation}) becomes
\begin{equation}\label{GRSequation}
\Ric (g) +\mathrm{Hess}(f) =\lambda g,
\end{equation}
and the metric $g$ is known as a \textit{gradient} Ricci soliton with potential function $f$.
Perelman \cite{Per1} demonstrated that, on compact manifolds, any non-Einstein Ricci soliton is a shrinking gradient Ricci soliton. Therefore we only consider shrinking gradient Ricci solitons on compact manifolds in the sequel.
\subsection{Ricci soliton moduli}
The study of moduli spaces for Ricci solitons was initiated by Podest\`a and Spiro in \cite{PS}.  It is clear from Equation (\ref{GRSequation}) that, given a gradient Ricci soliton $g$ with potential function $f$, adding a constant to $f$ yields admissible potential function.  Hence, in order to study moduli of solitons, the potential functions should be normalised.\\
\\
If $\mathcal{M}$ denotes the space of all Riemannian metrics on a manifold $M^{m}$, Podest\`a and Spiro introduce the set ${\mathcal{P}\subset\mathcal{M}\times C^{\infty}(M)}$ defined by
$$
\mathcal{P} = \left\{ (g,f) \ \bigg |  \ \int_{M}e^{-f}d\mathrm{Vol}_{g} = (2\pi)^{m/2} \right\}.
$$
They define the set $\mathcal{S}ol\subset\mathcal{P}$ to be the pairs solving Equation (\ref{GRSequation}) after fixing $\lambda=1$. Given a Ricci soliton ${(g,f) \in \mathcal{S}ol}$ they define an `$f$-twisted' slice inside $\mathcal{M}$ for the action of the diffeomorphism group and the premoduli space about $(g,f)$ is the intersection of the slice (or rather the product of the slice and $C^{\infty}(M)$) and the set $\mathcal{S}ol$. We will not directly  need the theory developed and so we refer the reader to the paper \cite{PS} for further details. There is also a solitonic notion of infinitesimal variation generalising that of EIDs given in Definition \ref{def:EID}. Podest\`a and Spiro use a different definition of the space of essential infinitesimal soliton deformations and then show the following is equivalent (see Theorem 3.3 in \cite{PS}).
\begin{definition}[Essential infinitesimal solitonic variation]
Let $(g,f)\in \mathcal{S}ol$ be a shrinking Ricci soliton and let ${h\in s^{2}(T^{\ast}M)}$  satisfy 
\begin{equation}{\label{h_tf_sol}}
\tr(h)=\alpha,
\end{equation}	
\begin{equation}{\label{h_df_sol}}
\mathrm{div}_{f}(h):=\mathrm{div}(h)-\iota_{\nabla f}h=0,
\end{equation}
and	
\begin{equation}{\label{h_LL_sol}}
\Delta_{f} h +2\mathrm{Rm}(h) =0,
\end{equation}	
where $\alpha \in C^{\infty}(M)$ and ${\Delta_{f}(\cdot) := \Delta(\cdot)-\nabla_{\nabla f}(\cdot)}$. Then the pair $(h,a)$ is referred to as an (essential) \textit{ infinitesimal solitonic deformation} (ISD). We denote the space of such $(h,\alpha)$ by $\sigma(g,f)$.
\end{definition}
\noindent In \cite{PS}, Podest\`a and Spiro prove the solitonic analogue of Koiso's Theorem \ref{KST}: given $(g,f)\in \mathcal{S}ol$, there is a real analytic submanifold $\mathcal{Z}$ of the product of the $f$-twisted slice and $C^{\infty}(M)$ such that the tangent space to $Z$ at $(g,f)$ is $\sigma(g,f)$ and the premoduli space for $(g,f)$ is  a real-analytic subset of  $\mathcal{Z}$.\\
\\
If the Ricci soliton is an Einstein metric then the potential function $f$ is constant and
\[
f=\log\left(\vol(M,g)\right)-\frac{m}{2}\log\left(\frac{2\pi}{\lambda}\right).
\] 
In this case the space of infinitesimal solitonic deformations decomposes as the direct sum of the infinitesimal Einstein deformations and an eigenspace for a particular eigenvalue of the Laplacian on functions. 
\begin{lemma}[Podesta--Spiro, Proposition 4.4 in \cite{PS}  - see also Kr\"oncke Lemma 6.2 in \cite{KKCVP}] \label{lem:KlausISD}
Let $(M, g)$ be an Einstein manifold with Einstein constant $\lambda$. Let EID be the space of infinitesimal Einstein deformations and ISD the space of infinitesimal solitonic deformations. Then ISD has the following decomposition:
$$
\mathrm{ISD}=\mathrm{EID} \oplus \left\{ \lambda v\cdot g+\mathrm{Hess}(v) \bigg | v\in C^{\infty}(M), \Delta v=-2\lambda v \right\}.
$$
\end{lemma} 
\noindent We will denote the $-2\lambda$-eigenspace of the Laplacian by $V_{-2\lambda}$.   As with infinitesimal Einstein deformations, a natural question is whether an element of ISD is the tangent to a genuine family of Ricci solitons.  Kr\"oncke proved the following analogy to Lemma \ref{Koilem1}. 
\begin{theorem2}[Kr\"oncke, Theorem 5.7 in \cite{KKJGA}] \label{thm:KK_Int}
Let $(M,g)$ be an Einstein metric with Einstein constant $\lambda>0$ and  let 
$v\in V_{-2\lambda}$. Then 
\[h = \lambda v\cdot g+\nabla^{2}v, \]
where $v\in \mathrm{ISD}$, is not integrable to second order if there exists a function $w \in V_{-2\lambda}$  such that
\[
\int_{M}v^{2}w \ d\mathrm{Vol}_{g} \neq 0. 
\]
\end{theorem2}
\noindent An equivalent formulation of this theorem is to say that, if $v\in V_{-2\lambda}$, then the corresponding ISD $h$ is integrable to second order if and only if the $L^{2}$-projection of $v^{2}$ to $V_{-2\lambda}$ vanishes.\\
\\
Let $M=G/K$ be a compact irreducible Hermitian symmetric space; a classical result due to Matsushima \cite{Mat} gives a $G$-equivariant isomorphism $\Phi:\mathfrak{g}\rightarrow V_{-2\lambda}$. The methods of Section \ref{sec:SU(n)calcs} thus give a proof of the following theorem of Kr\"oncke.
%\setcounter{theorem}{2}
%\begin{theorem}[Kr\"oncke, (Theorem 1.3 in \cite{KKJGA})]\label{thm:KroenckeRigid}
%The Fubini--Study metric $g$ on $\mathbb{CP}^{2n}$ is solitonically rigid.  
%\end{theorem}
\begin{proof}[Proof of Theorem \ref{ThmC}] 
The space of infinitesimal Einstein deformations $\varepsilon(g)$ vanishes \cite{KoiOsaka2} and so $\mathrm{ISD} \cong \mathfrak{su}_{2n+1}$.\\
\\
 The obstruction given by Theorem \ref{thm:KK_Int} is an element of $\mathrm{Hom}_{SU_{2n+1}}(s^{2}(\mathfrak{g}),\mathfrak{g})$, where ${\mathfrak{g} = \mathfrak{su}_{2n+1}}$ which is one dimensional by Lemma \ref{lem:su_n_hom}. The obstruction can be understood to be a multiple of the map 
 $$
 \Psi(X,Y) = \sqrt{-1}\left(XY+YX - \frac{\tr(XY+YX)}{2n+1}\mathrm{Id_{2n+1}}\right),
 $$
 and as $2n+1$ is odd, this map  does not have any non-trivial zeroes.  It is possible to find an eigenfunction $v\in V_{-2\lambda}$ such that
\[
\int_{\mathbb{CP}^{2n}} v^{3}d\mathrm{Vol}_{g} \neq 0, 
\]
(see \cite{HMW}, \cite{KS}, or \cite{KKCAG}). Hence the obstruction is a non-zero multiple of $\Psi$ and so does not vanish for any perturbation induced by $v\in V_{-2\lambda}$.
\end{proof}
\noindent Podest\`a and Spiro also considered the following notion of rigidity for an Einstein metric.

\begin{definition}[Weak solitonic rigidity - Definition 5.1 in \cite{PS}]
A compact Einstein metric $g$ is said to be weakly solitonic rigid if there is a neighbourhood $\mathcal{U}\subset \mathcal{M}$ of $g$ such that every Ricci soliton in $\mathcal{U}$ is Einstein.
\end{definition}
\noindent We can now prove Theorem \ref{ThmB} which is a generalisation of Theorem \ref{ThmC} in this weaker setting.

\begin{proof}[Proof of Theorem \ref{ThmB}] 
The structure theorem of Podest\`a--Spiro for the solitonic premoduli space guarantees that in a small enough neighbourhood of $g$, all solitons come from integrable deformations of $g$.\\
\\
If $N=2n+1$, the space of solitonic deformations on ${M=SU_{N}/S(U_{k}\times U_{N-k})}$ splits as in Lemma \ref{lem:KlausISD}.  As in the proof of  Theorem \ref{ThmC}, the eigenspace $V_{-2\lambda}$ is $SU_{2n+1}$-equivariantly isomorphic to $\mathfrak{su}_{2n+1}$.  In \cite{HMW}, the authors show that there exists an eigenfunction $f \in V_{-2\lambda}$ with
\[
\int_{M} f^{3}d\mathrm{Vol}_{g}\neq 0.
\]
Hence Kr\"oncke's obstruction in Theorem \ref{thm:KK_Int} does not vanish and is given by a multiple of the map $\Psi$ from the proof of Theorem \ref{ThmB} which has no non-trivial zeroes on $\mathfrak{su}_{2n+1}$.  Hence the only possibly integrable deformations are EIDs.
\end{proof}
\section{Applications to stability of the Ricci flow and to other rigidity problems}\label{sec:AOB}
\subsection{Applications to dynamical stability}
\noindent  For a general compact Riemannian manifold $(M^{m},g)$, Perelman's $\nu$-entropy \cite{Per1} is given by
\[
\nu(g)  = \inf \left\{ \mathcal{W}(g,f,\tau): f\in C^{\infty}(M), \ \tau>0, \ (4\pi\tau)^{-\frac{m}{2}}\int_{M}e^{-f}d\mathrm{Vol}_g =  1 \right \},
\]
where
\[
\mathcal{W}(g,f,\tau) = \int_{M}[\tau (\mathrm{S}(g)+|\nabla f|^{2})+f-m]e^{-f}d\mathrm{Vol}_g.
\]
For any metric $g$, the infimum is always achieved by a pair $(f,\tau)$.   Perelman demonstrated the following formula for the first variation of $\nu$; let ${g(s):=g+sh}$ for some $h\in s^{2}(T^{\ast}M)$ and $s\in (-\varepsilon,\varepsilon)$ then  
\begin{equation}\label{eqn:1var_nu}
\dfrac{d}{ds}\bigg|_{s=0} \nu(s) = -\tau(4\pi\tau)^{-\frac{m}{2}} \int_{M} \bigg\langle  h, \ \Ric(g) + \Hess(f) - \frac{1}{2\tau}g \bigg\rangle e^{-f} d\mathrm{Vol}_g.
\end{equation}
What Equation (\ref{eqn:1var_nu}) demonstrates is that Ricci solitons are critical points of $\nu$. Furthermore, from  Equation (\ref{eqn:1var_nu}), it can be easily deduced that the $\nu$-entropy is increasing   along the Ricci flow except at these critical points; hence, variations that increase the entropy are destabilising.\\
\\

We return to the case where $g$ is an Einstein metric but, in order to keep our formulae in line with others in the literature, we will write the Einstein constant $\lambda=\frac{1}{2\tau}$. The second variation variation of $\nu$ at an Einstein metric was computed by Cao, Hamilton and Ilmanen \cite{CHI} (see \cite{CZ} for a proof in the general case of a Ricci soliton). They showed that the second variation on TT tensors is controlled by the spectrum of the Lichnerowicz Laplacian and that the second variation vanishes if $h$ is an EID. Thus, to determine whether an EID $h$ is destabilising, we must compute the third derivative of $\nu(g(s))$ at $s=0$.  To do this we will follow the method used by  Knopf and Sesum \cite{KS} to compute the third variation in conformal directions and write
\begin{align*}
\dfrac{d \nu(s)}{ds}\bigg|_{s=0} &= A  \int_M \bigg\langle  h, \ \Ric(g) + \Hess(f) - \frac{1}{2\tau} g \bigg\rangle e^{-f} d\mathrm{Vol}_g\\
&= A\int_M \underbrace{\bigg(g^{ip}g^{jq}h_{pq}\bigg)}_{B} \underbrace{\bigg( \Ric_{ij} + \Hess(f)_{ij} - \frac{1}{2\tau}g_{ij}\bigg)}_{C}\underbrace{\bigg( e^{-f}d\mathrm{Vol}_g\bigg)}_{D},
\end{align*}
where $A = -\tau(4\pi\tau)^{-\frac{m}{2}}$.  In the following proof we will also use the convention of  \cite{KS} and write $\mathcal{A}', \mathcal{A}'', etc.$ for the derivative with respect to $s$ and evaluated at $s=0$ of any quantity $\mathcal{A}$. 
 
\begin{proof}[Proof of Theorem \ref{ThmD}] 
Clearly, when $s=0$ we have that  $C=0$.  Hence Lemma 2.4 of \cite{CZ} implies that $\tau'=0$ and so $A'=0$. Using the formulae for $\Ric'$ (see \cite{Top} for example), the fact that $h$ is an EID and that $f(0)$ is constant, we have
\[
C' = -\Hess(f').
\]
Cao and Zhu give a characterisation of $f'$ for variations at general Ricci soliton (see the proof of Lemma 2.4 in \cite{CZ}).  It is clear that at an Einstein metric if $h$ is TT then $f'=0$ and so $C'=0$. Elementary calculus then yields 
\begin{align*}
\frac{d^{3}}{ds^{3}}\nu(g(s))\bigg|_{s=0} &= A\int_{M}BC''D \\ 
&= -\tau(4\pi\tau)^{-\frac{m}{2}} \int_{M} \bigg\langle  h, \ \Ric(g)'' + (\Hess(f))'' \bigg\rangle e^{-f} d\mathrm{Vol}_g.
\end{align*}
As $f'=0$, 
\[
(\Hess(f))'' = \Hess(f'').
\]
The TT tensors are $L^{2}$-orthogonal to Lie derivatives of the metric and so, as $\Hess(f'') = \frac{1}{2}L_{\nabla f''}g$, we have
\[
\frac{d^{3}}{ds^{3}}\nu(g(s))\bigg|_{s=0} =-\tau(4\pi\tau)^{-\frac{m}{2}} \int_{M} \bigg\langle  h,\ \Ric(g)'' \bigg\rangle e^{-f} d\mathrm{Vol}_g.
\]
The result then follows from Proposition 4.2 and Lemma 4.3 in \cite{KoiOsaka2}.  
\end{proof}

\begin{proof}[Proof of Theorem \ref{ThmE}]
For any $n>2$ we have demonstrated an EID $h$ such that $\mathcal{I}(h)\neq 0$.  For the family of metrics $g(s):=g+sh$, the first and second derivatives of $\nu(g(s))$ but, by Theorem \ref{ThmD}, the third derivative of $\nu(g(s))$ does not vanish.  Hence $g$ is not a local maximum of the $\nu$-entropy and so $g$ is dynamically unstable. 
\end{proof} 
\subsection{Applications to other rigidity problems}
For the symmetric spaces $SU_{2n+1}/SO_{2n+1}$, if one can demonstrate that, for some $h_{\eta}$, the quantity $\mathcal{I}$ does not vanish then none of the EID on these spaces is integrable. It would be particularly interesting to compute $\mathcal{I}$ given Theorem \ref{ThmD}, as the dynamical stability of the five-dimensional space $SU_{3}/SO_{3}$ is currently unknown.\\
\\ 
Our methods show almost all the EIDs on $SU_{2n}$ are not integrable.  Higher order obstructions need to be considered for the remaining cases.\\
\\
Koiso demonstrated in \cite{KoiOsaka2} that $\mathrm{Hom}_{E_{6}}(s^{2}(\mathfrak{e}_{6}),\mathfrak{e}_{6})=0$ and so all the EIDs on this space are integrable to second order.\\
\\
The Grassmannians have EIDs, but these do not arise from the construction outlined in Section \ref{sec:deform}.  For Grassmannians $SU_{p+q}/S(U_p\times U_q)$, where  
\[
K = S(U_p\times U_q) \qquad \mathrm{and} \qquad \mathfrak{p} = \left(\begin{array}{cc}0_{p\times p} & Z \vspace{2pt} \\ -\overline{Z}^{t} & 0_{q\times q}\end{array}\right),
\]
for $Z\in M^{p\times q}(\mathbb{C})$. As the product of any three matrices in $\mathfrak{p}$ is trace-free, we see directly that the space $s^{3}(\mathfrak{p})^{K}$ vanishes in this case. The construction of infinitesimal variations for the Grassmannians  is detailed in Chapter VIII of the book \cite{GGbook}.  The method of generating EIDs uses some of the complex differential geometry specific to the case of the Grassmannians and in particular, the generalised Euler sequence. Again, for the spaces $SU_{2n+1}/S(U_{n-k}\times U_{k})$, if one such deformation is obstructed then they are all obstructed. 
We hope to investigate this construction in a future work.\\
\\
Finally, we note that the bi-invariant metric on the compact Lie group $G_{2}$ is known to admit infinitesimal solitonic deformations \cite{CH}. In \cite{SJHG_2}, the second author demonstrated that there exist deformations that are not integrable to second order and hence the Einstein metric is dynamically unstable. It would be interesting to extend this analysis to characterise precisely which solitonic deformations are not integrable.

\bibliographystyle{acm} 
\bibliography{BHMW20Refs}
\end{document}